\documentclass[11pt,a4paper,final]{amsart}
\usepackage[utf8]{inputenc}
\usepackage[T1]{fontenc}
\usepackage[UKenglish]{babel}
\usepackage[a4paper,margin=1in]{geometry}

\usepackage{amsmath,amsthm,amsfonts,amssymb}
\usepackage{mathrsfs,dsfont}
\usepackage{graphicx}
\usepackage{url}
\usepackage{verbatim}
\usepackage[normalem]{ulem}

\usepackage[dvipsnames]{xcolor}

\usepackage{import}
\usepackage{xifthen}
\usepackage{pdfpages}
\usepackage{transparent}

\usepackage{hyperref}
\usepackage{mathtools}

\usepackage{bbm}

\usepackage{hyperref}

\usepackage{marginnote}

\usepackage{soul} 
\setstcolor{magenta}

\usepackage[notcite,notref]{showkeys}

\graphicspath{ {./figs/} }

% ...
\makeatletter
\def\namedlabel#1#2{\begingroup
    #2%
    \def\@currentlabel{#2}%
    \label{#1}\endgroup
}
\makeatother

\theoremstyle{plain}
\newtheorem{theorem}{Theorem}[section]
\newtheorem{corollary}[theorem]{Corollary}
\newtheorem{lemma}[theorem]{Lemma}

\theoremstyle{definition}
\newtheorem{remark}[theorem]{Remark}
\newtheorem{example}[theorem]{Example}

\numberwithin{equation}{section}

\renewcommand\labelenumi{\textup{\alph{enumi})}}

\renewcommand\theenumi\labelenumi
\makeatletter\renewcommand{\p@enumii}{}\makeatother 

\renewcommand{\leq}{\leqslant}

\renewcommand{\geq}{\geqslant}

\newcommand{\loc}{\mathrm{loc}}

\DeclareMathOperator{\spec}{spec}

\DeclarePairedDelimiter\norm{\lVert}{\rVert}%

\newcommand{\R}{\mathds{R}}

\newcommand{\Ee}{\mathds{E}}

\newcommand{\nI}{\textup{I}}

\newcommand{\pr}{\mathds{P}}

\newcommand{\ex}{\mathds{E}}

\begin{document}
	
	\title[Integral kernels of Schrödinger semigroups]
	{Integral kernels of Schrödinger semigroups with nonnegative locally bounded potentials}
	
	
	\author[M.~Baraniewicz]{Miłosz Baraniewicz}
	\address[M.~Baraniewicz]{Faculty of Pure and Applied Mathematics\\ Wroc{\l}aw University of Science and Technology\\ ul. Wybrze{\.z}e Wyspia{\'n}skiego 27, 50-370 Wroc{\l}aw, Poland}
	\email{milosz.baraniewicz@pwr.edu.pl}
	
	\author[K.~Kaleta]{Kamil Kaleta}
	\address[K.~Kaleta]{Faculty of Pure and Applied Mathematics\\ Wroc{\l}aw University of Science and Technology\\ ul. Wybrze{\.z}e Wyspia{\'n}skiego 27, 50-370 Wroc{\l}aw, Poland}
	\email{kamil.kaleta@pwr.edu.pl}
	
	\thanks{Research supported by National Science Centre, Poland, grant no.\ 2019/35/B/ST1/02421}
	
	\begin{abstract}
		We give the upper and the lower estimates of heat kernels for Schr\"odinger operators $H=-\Delta+V$, with nonnegative and locally bounded potentials $V$ in $\R^d$, $d \geq 1$. We observe a factorization:\ the contribution of the potential is described separately for each spatial variable, and the interplay between the spatial variables is seen only through the Gaussian kernel -- optimal in the lower bound and  nearly optimal  in the upper bound. In some regimes we observe the exponential decay in time with the rate corresponding to the bottom of the spectrum of $H$. Our estimates identify in a fairly informative and uniform way the dependence of the potential $V$ and the dimension $d$. The upper estimate is more local; it applies to general potentials, including confining and decaying ones, even if they are non-radial, and mixtures. The lower bound is specialized to confining case, and the contribution of the potential is described in terms of its radial upper profile. Our results take the sharpest form for confining potentials that are comparable to radial monotone profiles with sufficiently regular growth  -- in this case they lead to two-sided qualitatively sharp estimates. In particular, we describe the large-time behaviour of \emph{nonintrinsically ultracontractive} Schr\"odinger semigroups -- this problem was open for a long time. The methods we use combine probabilistic techniques with analytic ideas. We propose a new effective approach which leads us to short and direct proofs.
		
\medskip
	

\medskip

\noindent
2010 {\it MS Classification}: 47D08, 60J65, 35K08.  \\
	\end{abstract}

	\maketitle
	
		\section{Introduction}

	We consider the Schr\"odinger operator $H=-\Delta+V$ acting in $L^2(\R^d,dx)$, $d \geq 1$, where the potential $V:\R^d \to [0,\infty)$ is a locally bounded function. The corresponding Schr\"odinger semigroup $\big\{e^{-tH}: t \geq 0\big\}$ consists of integral operators, i.e.\
	\[
	e^{-tH} f(x) = \int_{\R^d} u_t(x,y) f(y) dy, \quad f \in L^2(\R^d,dx), \ t>0.
	\]
It is known that $(0,\infty) \times \R^d \times \R^d \ni (t,x,y) \mapsto u_t(x,y)$ is a continuous function, symmetric in $(x,y)$. It satisfies the inequality $u_t(x,y) \leq g_t(x,y)$, where 
\[
g_t(x,y) = \frac{1}{(4 \pi t)^{d/2}} \exp\left(- \frac{|y-x|^2}{4t}\right)
\]
is the Gauss--Weierstrass kernel, the heat kernel of the free kinetic term $-\Delta$. Our standard references for Schr\"odinger operators and semigroups are the paper by Simon \cite{Simon} and the monographs by Reed and Simon \cite{Reed-Simon}, Demuth and van Casteren \cite{Demuth-Casteren}, and Chung and Zhao \cite{Chung-Zhao}.

The main goal of the paper is to find the upper and the lower estimates of the integral kernels $u_t(x,y)$, as uniform and informative as possible.  We want to understand the contribution of the potential and the relation to the heat kernel $g_t(x,y)$ of the kinetic term $-\Delta$. 

Estimates of Schr\"odinger heat kernels have been widely studied in the literature. For a class of nonnegative and locally bounded potentials, research concentrated mainly on two extremal situations:\ when the potential is \emph{confining} (i.e.\ $V(x) \to \infty$ as $|x| \to \infty$) and \emph{decaying} (i.e.\ $V(x) \to 0$ as $|x| \to \infty$). Note that confining potentials are locally bounded and \emph{unbounded at infinity}, while decaying ones are always bounded. 

For confining potentials known results include classical \emph{intrinsic ultracontractivity} estimates by Davies and 
Simon \cite{ Davies-book,Davies-Simon}, diagonal-on estimates for polynomial potentials obtained by Sikora \cite{Sikora}, and short-time bounds by Metafune, Pallara and Rhandi \cite{Metafune-Pallara-Rhandi}, Metafune and Spina \cite{Metafune-Spina}, and Spina \cite{Spina}. For decaying case we refer to the paper by Zhang \cite{Zhang-1}. See also further references in these papers.  It is known to be a  challenging problem to identify the actual contribution of the potential in a sharp and clear way. For both cases, the sharpest results were obtained for rather restricted classes of potentials, mainly of polynomial type. Moreover, when the potential is bounded or it grows at infinity not too fast, then  the connection with the original Gaussian kernel is expected as well. In that case we see the term $g_{ct}(x,y)$ in the estimates. The constant $c$, typically different in the lower and the upper bound, is often far from being optimal.  In connection to that, we want to mention here the recent papers by Bogdan, Dziubański and Szczypkowski \cite{Bogdan-Dziubanski-Szczypkowski}, and Jakubowski and Szczypkowski \cite{Jakubowski-Szczypkowski} which thoroughly analyze the problem of uniform comparability of the Schr\"odinger heat kernel with the original heat kernel $g_t(x,y)$ of the kinetic term $-\Delta$, and give sharp conditions for this property. We also refer to these papers for an extensive discussion of the related literature.

Theorems \ref{th:main-1} and \ref{th:lower_main}, the main results of this paper, give the estimates of the kernel $u_t(x,y)$ which are global in time and space.

 The upper estimate in Theorem \ref{th:main-1} applies to every nonnegative locally bounded potential, i.e.\ it allows one to analyze in a joint framework the functions $V$ of various types, including confining and decaying potentials, functions being bounded away from zero or even mixtures of all of these types.  For instance,  in one-dimensional case it covers highly nonsymmetric potentials which are confining on one halfline, and decaying or bounded on the other.  This seems to be a novelty.  Moreover, this bound matches, at least structurally and qualitatively, the lower estimate in Theorem \ref{th:lower_main} which is specialized to confining potentials or potentials with confining components in unbounded subsets of the space. 
It leads to qualitatively sharp two-sided estimates for confining potentials that are comparable to radial  monotone  profiles with sufficiently regular growth, see Examples \ref{ex:conf} and \ref{ex:conf_gen} for illustration. In particular, we give the large-time estimates for \emph{nonintrinsically ultracontractive} Schr\"odinger semigroups. This problem was solved for a large class of nonlocal operators \cite{Kaleta-Schilling}, but it was left open for classical Schr\"odinger semigroups.   

We remark that our estimates identify in a fairly informative and uniform way the dependence on the potential and the dimension of the space. Moreover, the Gaussian term is optimal (the lower bound) or it can be made nearly optimal (the upper bound). We propose a novel approach which leads us to rather short and direct proofs.     
To the best of our knowledge, these contributions are new. 

The structure of the paper is as follows. In Section \ref{sec:results} we present Theorems \ref{th:main-1} and \ref{th:lower_main}, discuss the structure of estimates, uniform rates and constants, and explain the ideas of proofs. Section \ref{sec:prelim} contains preliminaries. In Sections \ref{sec:upper} and \ref{sec:lower} we give the proofs of the upper bound and the lower bound, respectively. Finally, in Section \ref{sec:applications}, we discuss applications of our estimates and connections to the literature. In particular, 
Subsection \ref{subsec:conf} is devoted to analysis of confining potentials, and the two-sided estimates we get for this class. In Subsection \ref{subsec:dec} we illustrate the upper bound for  decaying potentials.  In the last Subsections \ref{subsec:bdd_away} and \ref{subsec:general} we analyze the examples of potentials that are bounded away from zero, and some more general cases, including mixtures and those  nonradial.  

\subsection*{Acknowledgement} We are grateful to Krzysztof Bogdan, Tomasz Jakubowski, Tadeusz Kulczycki and Ren\'e Schilling for discussions and useful comments which helped us to improve the presentation of our manuscript. We also thank Grzegorz Serafin for the discussion on the estimates of transition densities of the killed Brownian motion and related references. 

\subsection*{Notation}
The Euclidean ball centered at $x \in \R^d$, with radius $r>0$, will be denoted by $B_r(x)$. When $x=0$, we simply write $B_r$. We use the notation $a \wedge b := \min\left\{a,b\right\}$ and $a \vee b := \max\left\{a,b\right\}$.

\section{Presentation of results} \label{sec:results}

Our first main result gives the upper estimate in terms of the \emph{lower profile} of the potential:
\[ 
V_*(x):= \inf_{z \in \overline{B}_{|x|/2}(x)} V(z), \qquad x \in \R^d.
\]
We use the following notation:
		\[
		H(t,x):= \exp\left(-\frac{\sqrt{2}}{32} \left(\left(V_*(x)+\frac{\mu_0}{4|x|^2}\right)t \, \wedge \, 2|x| \sqrt{V_*(x)+\frac{\mu_0}{4|x|^2}} \, \right)\right), \quad x \neq 0,
\]
$H(t,0) := \exp\big(-\frac{\sqrt{2 \mu_0}}{32}\big)$, and
\[
\gamma_1:=\frac{\lambda_0}{2}.
\]
The number $0 \leq \lambda_0:= \inf \sigma(H)$ is the bottom of the spectrum of the Schr\"odinger operator $H$, and $\mu_0>0$ is the principal eigenvalue of the operator $-\Delta_{B_1}$, the (positive) Dirichlet Laplacian on the ball $B_1$. 

\begin{theorem} \label{th:main-1}	
There exists a constant $c_1=c_1(d)$ such that for every $0 \leq V \in L^{\infty}_{\loc}(\R^d)$, $x,y \in \R^d$ and $t >0$ we have
		\begin{equation*}
			u_t(x,y) \leq c_1 \, H(t,x) \, H(t,y) \, \left(g_{2t}(x,y) \wedge e^{-\gamma_1 t} g_{t}(0,0) \right).
		\end{equation*} 
\end{theorem}
	
\bigskip
	
\noindent
We note that $c_1$ can be chosen as
\[
c_1=2^{(3d+4)/2}  e^{\sqrt{\mu_0}/8} (C_0 \vee C),
\]
where the constants $C_0$ and $C$ (independent of $V$) come from \eqref{eq:C_0_def} and Lemma \ref{lem:Wendel} below.
	
The upper bound in Theorem \ref{th:main-1} is \emph{uniform} in $0 \leq V \in L^{\infty}_{\loc}(\R^d)$ in the sense that the rate $\sqrt{2}/32$ in $H(t,x)$ and the constant $c_1$ are independent of $V$.  This estimate  is a version of more general result, stated as Theorem \ref{th:upper_final}, which allows one to get the upper bound with the Gaussian term $g_{at}(x,y)$ with $a >1$ arbitrarily close to $1$, i.e.\ the Gaussian term can be made \emph{ nearly optimal }. Note, however, that this can be done at the cost of the absolute constant in the exponent of $H(t,x)$. For simplicity, here we just choose $a=2$. We remark that the rate $\sqrt{2}/32$ is not optimal, but we state it here explicitly in order to show the uniform structure of the estimate. 
For applications of Theorem \ref{th:main-1}, examples and further discussion we refer the reader to Section 
\ref{sec:applications}. We remark that it gives qualitatively sharp bounds for a large class of potentials.

Our second main result gives the lower estimate of the kernel $u_t(x,y)$ in terms of the following \emph{radial upper profile} of the potential
\[
V^*(r) := \sup\limits_{z \in \overline{B}_{2r}} V(z), \qquad r \geq 0.
\] 
In the statement and below we use the shorthand notation:
\[
t_{\rho} := \frac{\rho}{2\sqrt{V^{*}(\rho)+\frac{\mu_0}{4\rho^2}}}, \quad \rho \geq 1,
\]
\[
K(t,\rho):= \exp\left(-\frac{9}{4} \left( \left(V^{*}(\rho)+\frac{\mu_0}{4\rho^2}\right)t \, \wedge \,  2\rho \sqrt{V^{*}(\rho)+\frac{\mu_0}{4\rho^2}}\right)\right), 
\]
\smallskip
\[
\rho_x := |x| \vee 1 \qquad \text{and} \qquad \gamma_2:=d+V^{*}(1)+\frac{\mu_0}{4}.
\]
\begin{theorem} \label{th:lower_main} 
There exists $c_2 = c_2(d)$ such that for every $0 \leq V \in L^{\infty}_{\loc}(\R^d)$, $x, y \in \R^d$ and $t>0$ we have the following estimates.
\begin{itemize}
\item[(1)] If $4t_{\rho_x \vee \rho_y} \leq t$, then 
\begin{align*}
u_t(x,y) \geq c_2 e^{-\gamma_2 t} g_t(0,0) K(t, \rho_x) 
 K(t,\rho_y).
\end{align*}

\item[(2)] If $4t_{\rho_x \vee \rho_y} \geq t$, then
\begin{align*}
u_t(x,y) \geq c_2 K(t, \rho_x \vee \rho_y) 
 g_t(x,y). 
\end{align*}
\end{itemize}
\end{theorem} 

\noindent
The constant $c_2$ in the above theorem can be chosen as
\[
c_2= \frac{(\widetilde C/4)^4}{4^{d}\Gamma(d/2+1)^{3}},
\]
where $\widetilde C$ is the constant (independent of $V$) from the auxiliary estimate \eqref{eq:lower_dir_balls} below.

As before, the estimate in Theorem \ref{th:lower_main} is \emph{uniform} in $0 \leq V \in L^{\infty}_{\loc}(\R^d)$ -- the rate $9/4$ in $K(t,\rho)$ and the constant $c_2$ are independent of $V$, and the Gaussian term $g_t(x,y)$ in this bound is optimal.
This result is rather specialized to confining potentials for which it leads to the global estimate given in terms of the upper profile $V^{*}$ -- this estimate is the most informative for potentials with radial profiles. Together with Theorem \ref{th:main-1}, it also immediately gives the two-sided estimates, see Corollary \ref{cor:dec}. These bounds are qualitatively sharp for potentials that are comparable to radial monotone functions with sufficiently regular growth, see Examples \ref{ex:conf} and \ref{ex:conf_gen}. However, it still can be  useful  in less regular situations. For example, if $V$ is the potential with confining component in some unbounded subset $D$ of $\R^d$, then our result may give the sharp lower estimate for $x, y \in D$. 
On the other hand, due to appearance of the exponential time rate which is always present in our bound, the theorem seems to be useless for decaying potentials. In this case, it just leads to an obvious estimate $u_t(x,y) \geq e^{-Ct} g_t(x,y)$.

Observe that the estimate in Part (1) of Theorem \ref{th:lower_main} can be equivalently rewritten as 
\[
u_t(x,y) \geq c_2 \, e^{-\gamma_2 t} \, g_t(0,0)\, K(t,\rho_x \vee \rho_y) \, K(t,\rho_x \wedge \rho_y),
\]
and the term $K(t,\rho_x \vee \rho_y)$ is already determined by the assumption $4t_{\rho_x \vee \rho_y} \leq t$. 

Also, we note that in Part (2) we get the estimate with the kernel $K$ in which the rate $9/4$ is replaced by $1$. However, for more clarity, in the statement we keep the same rate in both Parts (1) and (2).

\subsection*{Structure of our estimates and the uniform rates and constants. }

Our estimates show a factorization:\ the terms $H$ and $K$ are separate for each spatial variable, and if $\lambda_0>0$, then the additional decay in time is present as well. The interaction between spatial variables is described by the Gaussian kernel -- optimal in the lower bound, and  nearly optimal  in the upper bound.   

The structure of $H$ and $K$ is exactly the same; for small spatial variables these functions are less than $1$ and bounded away from zero, uniformly in $t>0$. However, for large $z$ and $t>0$, they clearly show a competition between two factors 
\begin{align}\label{eq:competition}
\exp\left(-c \left(V(z)+\frac{\mu_0}{4|z|^2}\right)t\right) \qquad \text{and} \qquad \exp\left(-2c|z| \sqrt{V(z)+\frac{\mu_0}{4|z|^2}}\right).
\end{align}
At the technical side, the contribution of the potential is described by the profiles $V_*$ and $V^*$, and the constant $c$ is equal to $9/4$ in the lower bound, and $\sqrt{2}/32$ in the upper bound. Of course, these numerical constants are not optimal, but they are explicit and absolute numbers, independent of the potential $V$ and the dimension $d$; the dependence on $d$ is expressed through $\mu_0$ only, and  our proofs show that the position of the term $\mu_0/|x|^2$, at least in the first expression in \eqref{eq:competition},  is correct and sharp. This term plays an important role in the proof of the lower estimate. Of course, in the case of confining potentials, the expression $V(z)+\mu_0/(4|z|^2)$ can be replaced by $V(z)$ for large $z$'s, but this is possible only at the cost of the constant depending on $V$ and $d$ -- this would destroy the absoluteness of the  rate  $9/4$ appearing in the lower bound. Also the multiplicative constants $c_1, c_2$ are independent of $V$ ($c_1, c_2$ depend on $d$ only). In this sense our results are \emph{uniform} with respect to $V$. On the other hand, the time rates $\gamma_1$, $\gamma_2$, giving the estimates of $\lambda_0$ (in the confining case $\lambda_0$ is the ground state eigenvalue), are necessary dependent on the potential and the dimension.

The effect of the competition in \eqref{eq:competition} strongly depends on the type of the potential and the rate of its growth or decay at infinity.  Further discussion in Section \ref{sec:applications} will be divided into separate parts, including confining and decaying cases.  

\subsection*{A few words about the proofs.} \label{sebsec:proofs}
Our proofs are based on the Feynman--Kac representation of the Schr\"odinger semigroup $\big\{e^{-tH}: t \geq 0\big\}$ with respect to Brownian motion, some general estimates for the exit time of this process from a ball, and the recent sharp estimate for the semigroup of the corresponding killed process. We  develop  a new approach which allows us to find a rather direct and short argument.  

 Theorem \ref{th:main-1} follows from more general Theorem \ref{th:upper_final}. The proof of the latter result starts in a rather standard way, by dividing the expectation defining the kernel into two terms with respect to the exit time from a ball centered at the starting position $x$, with the radius proportional to the norm of $x$. Further, it uses the natural monotonicity properties of the exponential function, the H\"older inequality, the semigroup property (the Chapman--Kolmogorov property) and the specific structure of the Gauss--Weierstrass kernel. The key step in the proof is based on the observation that the Laplace transform of the exit time from a ball mentioned above, evaluated at $\lambda = V(x)$ (formally we see here the lower profile of the potential) takes the form of the second term appearing in \eqref{eq:competition}. We use for this the classical result of Wendel \cite{Wendel}. Interestingly, this is exactly the shape of the ground state $\varphi_0$ of the Schr\"odinger operator $H$ \cite{Simon-1, Carmona}. These steps are made in Lemmas \ref{lem:upper} and \ref{lem:Wendel}.

We conclude the proof of Theorem \ref{th:upper_final} by symmetrizing the estimate obtained in Lemma \ref{lem:upper}, through the Chapman--Kolmogorov (semigroup) property -- this leads us to the factorization of the estimate in a very natural way. The estimate with the exponential correction in time follows from the fact that the $L^2$-norm of the operator $e^{-tH}$ is equal to $e^{-\lambda_0 t}$. We note, however, that in many situations $\lambda_0$ need not be a positive number. 

Proving Theorem \ref{th:lower_main} seems to be a more challenging problem, because we want to find the lower estimate matching the upper bound.  More precisely,  we want to get a factorization and keep the same structure of the terms appearing in the estimate. First, in Lemma \ref{lem:aux-1}, we prove a general estimate which covers Part (2) of the theorem. The structure of the upper bound and the estimate in this lemma suggest the structure of $t_{\rho_x}$ -- the space-dependent threshold time which decides about the shape of the estimate. We observe that the correct form of the estimate, matching the upper bound, is determined by the position of the time $t$ with respect to $t_{\rho_x \vee \rho_y}$. The key technical step is made in Lemma \ref{lem:aux-2} which allows us to get  Part (1)  of the theorem. The tricky estimates in the proof of this lemma shows in a clear way the play between the time and the spatial variables. One can say that the key idea used in the proof of the lower bound is to reduce the problem to estimating the semigroup of the process in a ball, without losing too much information. Interestingly, we do not use the joint distribution of the exit position and the exit time of the process from a ball. The only tool we need is the lower estimate for the kernel of the semigroup of the killed process with sharp Gaussian term, obtained just recently by Małecki and Serafin \cite{malecki-serafin}.

\section{Preliminaries} \label{sec:prelim}
Let $(X_t)_{t \geq 0}$ be the Brownian motion running at twice the speed, with values in $\R^d$, $d \geq 1$, over a probability space $(\Omega,\mathcal F, \pr)$. This is the stochastic process with continuous paths, starting from $0$, such that
\[
\pr(X_t \in dy) = \frac{1}{(4\pi t)^{d/2}} \exp\left(-\frac{|y|^2}{4t}\right)dy, \quad t>0,
\]
Note that the process $(X_t)_{t \geq 0}$ has the scaling property:\ $X_{at}$ has the same distribution as $
\sqrt{a} X_t$, $a>0$. 
We denote by $\pr_x$ the probability measure for the process starting from $x \in \R^d$, i.e.
\[
\pr_x(X_t \in dy) := \pr(X_t+x \in dy) = g_t(x,y) dy
\] 
where 
\[
g_t(x,y) = \frac{1}{(4\pi t)^{d/2}} \exp\left(-\frac{|y-x|^2}{4t}\right)
\]
and by $\ex_x$ -- the expected value with respect to $\pr_x$. We use the notation $\ex_x \left[F;\, A \right] = \int_A F d\pr_x$. 

Let $0 \leq V \in L^{\infty}_{\loc}(\R^d)$ and consider the Schr\"odinger operator $H=-\Delta+V$. The semigroup operators $e^{-tH}$, $t>0$, can be represented via the Feynman--Kac formula:
\[
e^{-tH} f(x) = \ex_x \left[\exp\left(-\int_0^t V(X_s)ds\right) f(X_t)\right], \quad f \in L^2(\R^d,dx),
\]
see Simon \cite{Simon-2}, Demuth and van Casteren \cite{Demuth-Casteren} or Chung and Zhao \cite{Chung-Zhao}. Consequently, the corresponding integral kernels $u_t(x,y)$ can be expressesd as
	\begin{align} \label{eq:lim_kernel}
		u_t(x,y) = \lim\limits_{s \nearrow t} \Ee_x \left[ e^{-\int_0^s V(X_u)du} g_{t-s}(X_s,y) \right], \quad x, y \in \R^d, \ t>0,
	\end{align}
see \cite[Proposition 2.7]{Demuth-Casteren}. In particular, $u_t(x,y) \leq g_t(x,y)$. 

Let $\tau_D := \inf\{t \geq 0: X_t \notin D\}$ be the first exit time of the process $(X_t)_{t \geq 0}$ from an open and bounded set $D \subset \R^d$. In this paper we mainly consider the case when $D=B_r$, $r>0$. 
Recall that the transition semigroup of the Brownian motion killed upon exiting $B_r$ consists of operators 
\[
G^{B_r}_t f(x) = \ex_x [f(X_t); t < \tau_{B_r}] = \int_{B_r} f(y) g^{B_r}_t(x,y)dy, \quad f \in L^2(B_r,dy), \ t>0,
\]
with continuous and bounded transition densities $g^{B_r}_t(x,y)$. Due to scaling property, it is sufficient to analyze the case $r=1$. All these operators are of Hilbert--Schmidt type; in particular, spectra of $ G^{B_1}_t$, $t>0$, consist of eigenvalues $e^{-\mu_n t}$, where the numbers $0<\mu_0 < \mu_1 \leq \mu_2 \leq \ldots \to \infty$ are eigenvalues of $-\Delta_{B_1}$, the (positive) Dirichlet Laplacian on the ball $B_1$. The operator $\Delta_{B_1}$ is the infinitesimal generator of the semigroup $\big\{G^{B_1}_t:\, t \geq 0 \big\}$. The number $\mu_0 := \inf \spec (-\Delta_{B_1})$ is called the ground state  (or principal)  eigenvalue; the corresponding eigenfunction $\phi_0 \in L^2(B_1,dx)$ is bounded, continuous and strictly positive on $B_1$.  

It is known that due to intrinsic ultracontractivity (see Davies and Simon \cite{Davies-Simon,Davies-book}), there exists a constant $c>0$ such that
\[
g^{B_1}_t(x,y) \leq c e^{-\mu_0 t} \phi_0(x) \phi_0(y), \quad x,y \in B_1, \ t \geq 1.
\]
By integrating on both sides of this inequality over $y \in B_1$, we get
\[
\pr_x(t < \tau_{B_1}) \leq c  e^{-\mu_0 t} \phi_0(x)\left\|\phi_0\right\|_1, \quad x \in B_1, \ t \geq 1.
\]
Consequently, we have the following estimate
\begin{align} \label{eq:C_0_def}
	\pr_0(t < \tau_{B_1}) \leq C_0 e^{-\mu_0 \, t}, \quad t>0,
	\end{align}
where $C_0:= e^{\mu_0} \vee c \, \phi_0(0) \left\|\phi_0\right\|_1 > 1$.
  	
	\section{The proof of the upper bound} \label{sec:upper}
	
Throughout this section we use the following constants:\ for $a>1$ we define
\[
		 C_1=C_1(a):= \left(2\left(C_0 \vee C\right)^{\frac{(a-1)}{a}} a^{\frac{d}{2}}\left(\frac{a}{a-1}\right)^{\frac{(a-1)d}{2a}}\right)^2
\]
(the constant $C_0$ comes from \eqref{eq:C_0_def}  above  and $C$ (independent of $a$ and $V$) is determined in Lemma \ref{lem:Wendel} below) 
and 
\[ 
		C_2 = C_2(a):= \frac{1}{2a}, \quad C_3=C_3(a):=\frac{2\mu_0(a-1)}{a^2},  \quad C_4=C_4(a):=\frac{1}{4}\sqrt{\frac{a-1}{a}}.
\]
We first prove the most general Theorem \ref{th:upper_final}, and then we come back to  Theorem \ref{th:main-1}. 
		\begin{theorem}\label{th:upper_final}
		Let $0 \leq V \in L^{\infty}_{\loc}(\R^d)$ and let $x,y \in \R^d$, $t >0$, $a>1$. Then 
		\begin{equation*}
			u_t(x,y) \leq C_1 \, h(t,x) \, h(t,y) \, g_{at}(x,y)
		\end{equation*}
		and
		\begin{equation*}
			u_t(x,y) \leq \sqrt{C_1}	\,  \frac{1}{(2a\pi t)^{d/2}} \, \exp\left(-\frac{\lambda_0}{2} t\right)  \,  \sqrt{h(t,x)} \, \sqrt{h(t,y)}, 
		\end{equation*}
		where 
		\[
		h(t,x):= \exp\left(-\left(\left(C_2 V_*(x)+\frac{C_3}{|x|^2}\right)t \, \wedge \, C_4 \sqrt{V_*(x)} \, |x|\right)\right). 
		\]
	\end{theorem}
	We use the convention that for $x=0$ we have $1/|x|^2 = +\infty$, so that $h(t,0) = 1$.
	The proof of the theorem will be given after a sequence of lemmas. We start with a remark.
	
	\begin{remark} \label{rem:upper_near_opt}
	The constants appearing in the estimates in Theorem \ref{th:upper_final} are not optimal, but they are explicit. In many cases, it is enough to take e.g.\ $a=2$. Nevertheless, depending on applications, one can choose the parameter $a$ to be arbitrarily close to $1$, which makes the Gaussian term $g_{at}(x,y)$ nearly optimal. Recall that the lower bound in Theorem \ref{th:lower_main} holds with $g_{t}(x,y)$. Note, however, that $C_3, C_4 \downarrow 0$ as $a \downarrow 1$.	
\end{remark}

	\begin{lemma}\label{lem:Wendel}  There exists a constant $C = C(d) \geq 1$ such that for every $\lambda \geq 0$ and $r>0$ we have 
		\begin{equation*}
			 \ex_0 \left[e^{-\lambda \tau_{B_r}}\right]  \leq C e^{-\frac{\sqrt{\lambda}r}{2}}.
		\end{equation*}
	\end{lemma}

	\begin{proof}
	First note that for $\lambda = 0$ the equality is trivial. Assume that $\lambda>0$. If $d=1$, then the assertion simply follows from the classical formula for the Laplace exponent of the first passage time for the one-dimensional Brownian motion, see e.g.\ \cite[Theorem 5.13]{SP12}. If $d \geq 2$, then we use the result of Wendel \cite[Eq.\ (4)]{Wendel} which says that 
	\begin{align}\label{eq:Wendel}
			\ex_0 \left[e^{-\lambda \tau_{B_r}}\right] =\frac{1}{2^{\frac{d-2}{2}}\Gamma\big(\frac{d}{2}\big)} \frac{\left(\sqrt{\lambda} r \right)^\frac{d-2}{2}}{I_{\frac{d-2}{2}}\left(\sqrt{\lambda}r\right)},
		\end{align}
		where 
	\begin{align} \label{eq:Bessel-1}
		I_{\nu}(u) = \sum_{k=0}^{\infty} \frac{\left(u/2\right)^{\nu+2k}}{k! \, \Gamma(\nu+k+1)}, \quad u, \nu \geq 0,
		\end{align}
		is the modified Bessel function of the first kind, see e.g.\ \cite[10.25.2]{NIST:DLMF}. Note that Wendel considered the standard Brownian motion while we work with the Brownian motion running at twice the speed, and $\tau_{B_r}$ is the exit time of our process. Therefore the right hand side of \eqref{eq:Wendel} is the original formula of Wendel with $r$ replaced by $r/\sqrt{2}$  (or, equivalently, with $\lambda$ replaced by $\lambda/2$). 
		
		It is known (see \cite[10.30.4]{NIST:DLMF}) that
		\[
		 \lim_{u \to \infty} \frac{I_\nu(u)\sqrt{2\pi u}}{e^u} = 1, \quad \nu \geq 0.
		\]
		This implies that there exists $u_0 = u_0(d) >0$ such that 
		\[
		\frac{u^\frac{d-2}{2}}{2^{\frac{d-2}{2}}\Gamma\big(\frac{d}{2}\big)I_{\frac{d-2}{2}}(u)} \leq e^{-u/2}, \quad u \geq u_0.
		\]
		On the other hand, by keeping only the term for $k=0$ in the series \eqref{eq:Bessel-1}, we get
		\[
		\frac{u^\frac{d-2}{2}}{2^{\frac{d-2}{2}}\Gamma\big(\frac{d}{2}\big)I_{\frac{d-2}{2}}(u)} \leq 1 \leq e^{u_0/2}e^{-u/2}, \quad u \in (0,u_0].
		\]
	 The assertion of the lemma follows then from the representation \eqref{eq:Wendel} with $C=e^{u_0/2}$. 
	\end{proof}

\begin{remark}
By direct inspection of the last lines of the proof, we can easily obtain the estimate of Lemma \ref{lem:Wendel} in the form:\ for every $\varepsilon \in (0,1)$ there exists a constant $C=C(d,\varepsilon) \geq 1$ such that
\begin{equation*}
			 \ex_0 \left[e^{-\lambda \tau_{B_r}}\right]  \leq C e^{-(1-\varepsilon)\sqrt{\lambda}r},
		\end{equation*}
i.e.\ the constant in the exponent can be arbitrarily close to $1$ at the cost of the multiplicative constant $C$.
For clarity, we keep the statement and the proof in the present, simple form.
\end{remark}

	\begin{lemma}\label{lem:upper}
		Let $0 \leq V \in L^{\infty}_{\loc}(\R^d)$ and let $x,y \in \R^d$, $t >0$, $a>1$. Then 
		\begin{equation*}
			u_t(x,y) \leq \sqrt{C_1}\exp\left(-\left(2\left(C_2 V_*(x)+\frac{C_3}{|x|^2}\right)t \, \wedge \, C_4 \sqrt{V_*(x)} \, |x|\right)\right) g_{at}(x,y).
		\end{equation*}
	\end{lemma}
	
	\begin{proof}
	Let $x,y \in \R^d$, $t>0$ and let $a>1$. If $x=0$, then $e^{-C_4\sqrt{V_*(x)}|x|}$=1. Consequently, 
	\[
	u_t(x,y) \leq g_t(x,y) \leq a^{d/2} g_{at}(x,y) = a^{d/2} e^{-C_4 \sqrt{V_*(x)}|x|} g_{at}(x,y)
	\]
	and the estimate trivially holds. Let $x \neq 0$ and denote $b=a/(a-1)$ so that $1/a+1/b=1$. 
	We have
	\begin{align*}
	u_t(x,y) = \int_{\R^d} u_{t/a}(x,z) u_{t/b}(z,y) dz & = \ex_x \left[e^{-\int_{0}^{t/a}V(X_s)ds} u_{t/b} (X_{t/a},y)\right] \\
& = \ex_x \left[e^{-\int_{0}^{t/a}V(X_s)ds} u_{t/b} (X_{t/a},y) ; \, t/a < \tau_{B_{{|x|/2}}(x)}\right] \\
& \ \ \ \ \ \ \ \ + \ex_x \left[e^{-\int_{0}^{t/a}V(X_s)ds} u_{t/b} (X_{t/a},y) ; \, t/a \geq \tau_{B_{|x|/2}(x)}\right] \\
& =: \nI_1 + \nI_2.
\end{align*}
Clearly,
\[
		\nI_1 \leq e^{-(t/a)V_*(x)} \ex_x \left[ g_{t/b}(X_{t/a},y);  \, t/a < \tau_{B_{{|x|/2}}(x)}\right] 
\]
and
\[
	\nI_2 \leq \ex_x \left[e^{-\int_0^{\tau_{B_{|x|/2}(x)}}V(X_s)ds}g_{t/b}(X_{t/a},y)\right] \leq \ex_x \left[e^{- V_*(x)\tau_{B_{|x|/2}(x)}}g_{t/b}(X_{t/a},y) \right],
\]
and, by Hölder's inequality, we get 
	\begin{equation}\label{eq:integrals1}
		\nI_1 \leq  e^{-(t/a)V_*(x)} \pr_x\big(t/a < \tau_{B_{{|x|/2}}(x)}\big)^{\frac{1}{b}} \ex_x \left[(g_{t/b}(X_{t/a},y))^{a}  \right]^{\frac{1}{a}},
	\end{equation} 
		\begin{equation}\label{eq:integrals2}
		\nI_2 \leq	\ex_x \left[e^{-bV_*(x)\tau_{B_{|x|/2}(x)}} \right]^{\frac{1}{b}}
			\ex_x \left[(g_{t/b}(X_{t/a},y))^{a}  \right]^{\frac{1}{a}}.
		\end{equation}
By space homogeneity and scaling of the Brownian motion, and the estimate \eqref{eq:C_0_def}, we obtain
\[
\pr_x\big(t/a < \tau_{B_{{|x|/2}}(x)}\big)^{\frac{1}{b}} = \pr_0\left(\frac{4t}{a|x|^2} < \tau_{B_1}\right)^{\frac{1}{b}} \leq C_0^{1/b} \exp\left(-\frac{4t}{ab|x|^2} \mu_0\right).
\]
Similarly, Lemma \ref{lem:Wendel} applied to $\lambda = bV_*(x)$ and $r = |x|/2$ gives the estimate
\[
\ex_x \left[e^{-bV_*(x)\tau_{B_{|x|/2}(x)}} \right]^{\frac{1}{b}} = \ex_0 \left[e^{-bV_*(x)\tau_{B_{|x|/2}(0)}} \right]^{\frac{1}{b}}
                \leq C^{1/b} e^{-\frac{\sqrt{V_*(x)} \, |x|}{4\sqrt{b}}}
\]
Moreover,
\begin{align*}
	\ex_x \left[(g_{t/b}(X_{t/a},y))^{a}  \right]^{\frac{1}{a}} & = \left(\int_{\R^d} g_{t/a}(x,z)g_{t/b}(z,y)\big(g_{t/b}(z,y)\big)^{a-1}  dz \right)^{1/a}\\            & \leq \left(\frac{b}{4 \pi t}\right)^{\frac{(a-1)d}{2a}} \left(\int_{\R^d} g_{t/a}(x,z) g_{t/b}(z,y)  dz \right)^{1/a} \\
	            & =  \left(\frac{b}{4 \pi t}\right)^{\frac{(a-1)d}{2a}} \left(g_{t}(x,y) \right)^{1/a}
							 = a^{\frac{d}{2}}b^{\frac{(a-1)d}{2a}} g_{at}(x,y). 
\end{align*}
With these estimates we can now come back to \eqref{eq:integrals1} and \eqref{eq:integrals2} and write
\[
\nI_1 \leq  C_0^{(a-1)/a} a^{\frac{d}{2}}\left(\frac{a}{a-1}\right)^{\frac{(a-1)d}{2a}}  \exp\left(-\frac{t}{a}V_*(x) - \frac{4(a-1)t}{a^2|x|^2} \mu_0\right)g_{at}(x,y),
\]
\[
\nI_2 \leq C^{(a-1)/a} a^{\frac{d}{2}} \left(\frac{a}{a-1}\right)^{\frac{(a-1)d}{2a}} \exp\left(-\frac{1}{4}\sqrt{\frac{a-1}{a}}\sqrt{V_*(x)} \, |x|\right) g_{at}(x,y).
\]
We conclude the proof by putting together the estimates for the expectations $\nI_1$ and $\nI_2$.
\end{proof}

	\begin{proof}[Proof of Theorem \ref{th:upper_final}]
	Let $x,y,z \in \R^d$, $t >0$ and $a>1$. First observe that by Lemma \ref{lem:upper} and symmetry of the kernel, we have
		\begin{equation*}
			u_{t/2}(x,z) \leq \sqrt{C_1}\exp\left(-\left(\left(C_2 V_*(x)+\frac{C_3}{|x|^2}\right)t \, \wedge \, C_4 \sqrt{V_*(x)} \, |x|\right)\right) g_{(a/2)t}(x,z),
		\end{equation*}
		and 
		\begin{equation*}
			u_{t/2}(z,y) \leq \sqrt{C_1}\exp\left(-\left(\left(C_2 V_*(y)+\frac{C_3}{|y|^2}\right)t \, \wedge \, C_4 \sqrt{V_*(y)} \, |y|\right)\right) g_{(a/2)t}(z,y).
		\end{equation*}
		The first estimate of the theorem follows directly from these bounds and the Chapman--Kolmogorov identity $u_t(x,y) = \int_{\R^d}u_{t/2} (x,z) u_{t/2} (z,y) dz$. 
		
		One more use of the Chapman--Kolmogorov property and the symmetry of the kernel, and the Cauchy--Schwarz inequality give us 
		\begin{align*}
			u_t(x,y) & = \int_{\R^d}u_{t/2} (x,z) u_{t/2} (y,z) dz \\ & \leq \left(\int_{\R^d} u_{t/2}^2(x,z) dz \right)^{\tfrac{1}{2}} \left(\int_{\R^d} u_{t/2}^2(y,z) dz \right)^{\tfrac{1}{2}}
		             = \Big(u_t(x,x) \Big)^{\tfrac{1}{2}} \Big(u_t(y,y) \Big)^{\tfrac{1}{2}}
		\end{align*}
		and 
		\begin{equation*}
			u_t(x,x) = \int_{\R^d} u_{t/2}^2(x,z) dz  = \norm{e^{-(t/4)H} u_{t/4}(x,\cdot)}_2^2 \leq \norm{e^{-(t/4)H}}_{2,2}^2 \, \norm{u_{t/4}(x, \cdot)}_2^2 = e^{-\frac{\lambda_0}{2}t} \, u_{t/2} (x,x).
		\end{equation*}
		Consequently,
			\begin{align} \label{eq:aux_up_1}
			u_t(x,y) \leq \Big(u_{t/2}(x,x) \Big)^{\tfrac{1}{2}} \Big(u_{t/2}(y,y) \Big)^{\tfrac{1}{2}} e^{-\frac{\lambda_0}{2}t}.
		\end{align}
		Now, from Lemma \ref{lem:upper} we have
		\begin{equation*}
			u_{t/2}(x,x) \leq \sqrt{C_1}\exp\left(-\left(\left(C_2 V_*(x)+\frac{C_3}{|x|^2}\right)t \, \wedge \, C_4 \sqrt{V_*(x)} \, |x|\right)\right) g_{(a/2)t}(x,x).
		\end{equation*}
		By applying this estimate to the both diagonal terms on the right hand side of \eqref{eq:aux_up_1}, we obtain the second claimed bound of the theorem.
	\end{proof}

\begin{proof}[Proof of Theorem \ref{th:main-1}] 
 By setting $a=2$ in Theorem \ref{th:upper_final}, we have $C_2= \frac{1}{4}$, $C_3 = \frac{\mu_0}{2}$, $C_4 = \frac{\sqrt{2}}{8}$ and, consequently, 

\[
h(t,x) \leq \exp\left(-\frac{\sqrt{2}}{16} \left(\left(V_*(x)+\frac{\mu_0}{4|x|^2}\right)t \, \wedge \, 2 |x| \sqrt{V_*(x)} \, \right)\right)
\]
and
\[
\sqrt{h(t,x)} \leq \exp\left(- \frac{\sqrt{2}}{32}  \left(\left(V_*(x)+\frac{\mu_0}{4|x|^2}\right)t \, \wedge \, 2 |x| \sqrt{V_*(x)} \, \right)\right).
\]
Moreover, recall that
\[
		H(t,x)= \exp\left(- \frac{\sqrt{2}}{32}  \left(\left(V_*(x)+\frac{\mu_0}{4|x|^2}\right)t \, \wedge \, 2 |x| \sqrt{V_*(x)+\frac{\mu_0}{4|x|^2}} \, \right)\right),  \quad x \neq 0,
\]
$ H(t,0) = \exp\big(-\frac{\sqrt{2 \mu_0}}{32}\big)$.
We will prove that
\begin{align} \label{eq:h_by_H}
 \sqrt{h(t,x)}  \leq \exp\left( \frac{\sqrt{\mu_0}}{16}\right) H(t,x).
\end{align}
Once this is done, the assertion follows from Theorem \ref{th:upper_final} with $a=2$ and the obvious inequality $h(t,x) \leq \sqrt{h(t,x)} $. Since  $\sqrt{h(t,0)}=1 \leq \exp{\left( \frac{(2-\sqrt{2})\sqrt{\mu_0}}{32} \right)} = \exp\left( \frac{\sqrt{\mu_0}}{16}\right) H(t,0)$, \eqref{eq:h_by_H} holds for $x=0$. We assume that $x \neq 0$ and consider two cases.

If $V_*(x) \geq \mu_0/(4|x|^2)$, then by the Taylor approximation,
\[
0 < |x|\left(\sqrt{V_*(x)+\frac{\mu_0}{4|x|^2}} - \sqrt{V_*(x)}\right) 
       \leq \frac{|x|}{2 \sqrt{V_*(x)}} \frac{\mu_0}{4|x|^2}
			 \leq \frac{|x|}{2 \sqrt{\frac{\mu_0}{4|x|^2}}} \frac{\mu_0}{4|x|^2} = \frac{\sqrt{\mu_0}}{4}.
\]
This gives that
\[
\exp\left(- \frac{\sqrt{2}}{16}  |x| \sqrt{V_*(x)} \, \right) \leq \exp\left( \frac{\sqrt{\mu_0}}{32}\right)\exp\left(- \frac{\sqrt{2}}{16}  |x| \sqrt{V_*(x)+\frac{\mu_0}{4|x|^2}} \, \right).
\]
On the other hand, if $V_*(x) \leq \mu_0/(4|x|^2)$, then
\[
\exp\left( \frac{\sqrt{2}}{16}  |x| \sqrt{V_*(x)+\frac{\mu_0}{4|x|^2}} \, \right) \leq \exp\left( \frac{\sqrt{2}}{16}  |x| \sqrt{\frac{\mu_0}{2|x|^2}} \, \right) = \exp\left( \frac{\sqrt{\mu_0}}{16}\right),
\]
and, consequently, 
\begin{align*}
\exp \left(- \frac{\sqrt{2}}{16}  |x| \sqrt{V_*(x)} \, \right) \leq 1 
 \leq \exp\left( \frac{\sqrt{\mu_0}}{16}\right) \exp\left(- \frac{\sqrt{2}}{16} |x| \sqrt{V_*(x)+\frac{\mu_0}{4|x|^2}} \, \right).
\end{align*}
These estimates show that \eqref{eq:h_by_H} holds true.
\end{proof}
	
	\section{The proof of the lower Bound} \label{sec:lower}

The proof of Theorem \ref{th:lower_main} consists of several auxiliary results. The only technical tool we use in our reasoning is the lower estimate with sharp Gaussian term of the density of the Brownian motion killed upon exiting a ball. This bound has been obtained just recently for small times by Małecki and Serafin \cite{malecki-serafin} (see also the classical result for more general domains by Zhang \cite{Zhang} and the newest paper for convex domains by Serafin \cite{Serafin}). By combining it with the classical intrinsic ultracontractivity estimate (see Davies and Simon \cite{Davies-book, Davies-Simon}) and by using the scaling property, one gets the estimate in the following form which is useful for our purposes below:\ 
there exists a constant $\widetilde C \in (0,1]$ such that
\begin{align} \label{eq:lower_dir_balls}
				g^{B_r}_t(x,y) \geq \widetilde C \frac{1 \wedge \frac{(r-|x|) (r-|y|)}{t}}{(1 \wedge \frac{r^2}{t})^{(d+2)/2}} \exp\left(- \mu_0 \frac{t}{r^2}\right) g_{t}(x,y), \quad r>0, \ x,y \in B_r,\,  t>0,
\end{align}	
see \cite[Corollary 1]{malecki-serafin}. 

\begin{lemma}\label{lem:aux-1}
	Let $0 \leq V \in L^{\infty}_{\loc}(\R^d)$. For every $x, y \in \R^d$, $\rho>0$ such that $|x|, |y| \leq \rho$ and $t>0$ we have
		\begin{align}\label{eq:low-1}
			u_t(x,y) \geq (\widetilde C/4) \exp\left(-\left(V^{*}(\rho)+\frac{\mu_0}{4\rho^2}\right)t\right) g_t(x,y).
		\end{align}
\end{lemma}
	
\begin{proof}
We first observe that by \eqref{eq:lim_kernel}, for every $r>0$ and $x, y \in B_r$, we have
\begin{align*}
		u_t(x,y) &\geq   \lim\limits_{s \nearrow t} \ex_x \left[ e^{-\int_0^s V(X_u)du} g_{t-s}(X_s,y)\ ; \ s < \tau_{B_r} \right] \\
	  & \geq e^{-t \, \sup_{z \in B_r} V(z)}  \lim\limits_{s \nearrow t} \ex_x \left[ g^{B_r}_{t-s}(X_s,y)\ ; \ s < \tau_{B_r} \right] \\
		& = e^{-t \, \sup_{z \in B_r} V(z)}  g^{B_r}_t(x,y).
\end{align*} 
By taking $r=2\rho$, we get from \eqref{eq:lower_dir_balls} that
\begin{align}\label{eq:lem-aux-1-1}
g^{B_r}_t(x,y) \geq \widetilde C \frac{1 \wedge \frac{\rho^2}{t}}{(1 \wedge \frac{4\rho^2}{t})^{(d+2)/2}} \exp\left(- \mu_0 \frac{t}{4\rho^2}\right) g_t(x,y) \geq (\widetilde C/4) \exp\left(- \mu_0 \frac{t}{4\rho^2}\right) g_{t}(x,y)
\end{align}
and, consequently,
\[
u_t(x,y) \geq (\widetilde C/4) \exp\left(-\left(V^{*}(\rho)+\frac{\mu_0}{4\rho^2}\right)t\right) g_t(x,y).
\]
\end{proof}	
	
Recall the notation:
\[
t_{\rho} = \frac{\rho}{2\sqrt{V^{*}(\rho)+\frac{\mu_0}{4\rho^2}}}, \quad \rho \geq 1, \qquad  \text{and} \qquad \rho_x = |x| \vee 1.
\]	
	
\begin{lemma}\label{lem:aux-2} 	Let $0 \leq V \in L^{\infty}_{\loc}(\R^d)$.
Then for every $x, y \in \R^d$, $t>0$ such that $\rho_y  \leq \rho_x$ and $t \geq 2t_{\rho_x}$
we have
	 \begin{align*}
			u_t(x,y)  \geq \frac{(\widetilde C/4)^2}{4^{d/2}\Gamma(d/2+1)(4\pi t)^{d/2}} & \exp\left(-\frac{dt}{2}\right) \exp\left(-\left(V^{*}(\rho_y)+\frac{\mu_0}{4\rho_y^2}\right)t\right) \\ & \times \exp\left(- \frac{9}{2} \rho_x \sqrt{V^{*}(\rho_x)+\frac{\mu_0}{4\rho_x^2}}\right).
		\end{align*}
	\end{lemma}
	
	\begin{proof}
		Since $t > t_{\rho_x}$, we can write
		\begin{align}\label{eq:eq1}
			u_t(x,y) & = \ex_x \left[e^{-\int_{0}^{t_{\rho_x}}V(X_s)ds} u_{t-t_{\rho_x}} (X_{t_{\rho_x}},y)\right] \nonumber \\
			         & \geq \ex_x \left[e^{-\int_{0}^{t_{\rho_x}}V(X_s)ds} u_{t-t_{\rho_x}} (X_{t_{\rho_x}},y) : t_{\rho_x} < \tau_{B_{2\rho_x}}\right] \nonumber \\
							 & \geq e^{-t_{\rho_x} V^*(\rho_x)} \ex_x \left[u_{t-t_{\rho_x}} (X_{t_{\rho_x}},y) : t_{\rho_x} < \tau_{B_{2\rho_x}}, \,X_{t_{\rho_x}} \in B_{\rho_y} \right].
		\end{align}
By Lemma \ref{lem:aux-1}, on the set $\big\{X_{t_{\rho_x}} \in B_{\rho_y}\big\}$, we have
		\begin{align*}
			u_{t-t_{\rho_x}}(X_{t_{\rho_x}},y) 
			  & \geq (\widetilde C/4) \exp\left(-\left(V^{*}(\rho_y)+\frac{\mu_0}{4\rho_y^2}\right)t\right) g_{t-t_{\rho_x}}(X_{t_{\rho_x}},y) \\
				& = \frac{\widetilde C/4}{(4\pi(t-t_{\rho_x}))^{d/2}} \exp\left(-\left(V^{*}(\rho_y)+\frac{\mu_0}{4\rho_y^2}\right)t\right) 
				    \exp\left(-\tfrac{|y-X_{t_{\rho_x}}|^2}{4(t-t_{\rho_x})}\right)
			\end{align*}				
and, because $t_{\rho_x} \leq t-t_{\rho_x} \leq t$ and $\rho_y \leq \rho_x$, this estimate can be continued as follows
	\begin{align*}
				& \geq \frac{\widetilde C/4}{(4\pi t)^{d/2}} \exp\left(-\left(V^{*}(\rho_y)+\frac{\mu_0}{4\rho_y^2}\right)t\right) 
				    \exp\left(-\tfrac{4\rho_y^2}{4 t_{\rho_x}}\right) \\
			  & \geq \frac{\widetilde C/4}{(4\pi t)^{d/2}} \exp\left(-\left(V^{*}(\rho_y)+\frac{\mu_0}{4\rho_y^2}\right)t\right) 
				    \exp\left(-\tfrac{\rho_x^2}{t_{\rho_x}}\right) \\
				&	= \frac{\widetilde C/4}{(4\pi t)^{d/2}} \exp\left(-\left(V^{*}(\rho_y)+\frac{\mu_0}{4\rho_y^2}\right)t\right) 
				    \exp\left(- 2 \rho_x \sqrt{V^{*}(\rho_x)+\frac{\mu_0}{4\rho_x^2}}\right).
		\end{align*}
With this bound we can come back to \eqref{eq:eq1} and write
	\begin{align*}
			u_t(x,y)  \geq \frac{\widetilde C/4}{(4\pi t)^{d/2}} & \exp\left(-\left(V^{*}(\rho_y)+\frac{\mu_0}{4\rho_y^2}\right)t\right)  \exp\left(- 2 \rho_x \sqrt{V^{*}(\rho_x)+\frac{\mu_0}{4\rho_x^2}}\right) \\
	      	& \times \exp\left(-\frac{\rho_x V^*(\rho_x)}{2\sqrt{V^{*}(\rho_x)+\frac{\mu_0}{4\rho_x^2}}}\right) \pr_x \left[ t_{\rho_x} < \tau_{B_{2\rho_x}}, \,X_{t_{\rho_x}} \in B_{\rho_y} \right].
		\end{align*}
We only need to estimate the last probability. By using \eqref{eq:lem-aux-1-1} and the inequality $t_{\rho_x} < t$, we get
\begin{align*} 
\pr_x \left[ t_{\rho_x} < \tau_{B_{2\rho_x}},\right.&\left.X_{t_{\rho_x}} \in B_{\rho_y} \right] \geq \int_{B_1} g^{B_{2\rho_x}}_{t_{\rho_x}}(x,z) dz \\
 & \geq \frac{\widetilde C |B_1|}{4(4\pi t_{\rho_x})^{d/2}} \exp\left(- \mu_0 \frac{t_{\rho_x}}{4\rho_x^2}\right) \exp\left(-\tfrac{4 \rho_x^2}{4t_{\rho_x}}\right) \\
 & \geq \frac{\widetilde C \pi^{d/2}}{4 \Gamma(d/2+1)(4\pi t)^{d/2}} \exp\left(- \mu_0 \frac{t_{\rho_x}}{4\rho_x^2}\right) \exp\left(-\tfrac{ \rho_x^2}{t_{\rho_x}}\right) \\
 & \geq \frac{\widetilde C}{4^{1+d/2} \Gamma(d/2+1) } \exp\left(-\frac{dt}{2} \right)\exp\left(- \mu_0\frac{t_{\rho_x}}{4\rho_x^2}\right) \exp\left(- 2 \rho_x \sqrt{V^{*}(\rho_x)+\frac{\mu_0}{4\rho_x^2}}\right),
\end{align*}
where in the last line we used the inequality $t^{-d/2} \geq e^{-(d/2)t}$.
This leads us to the conclusion:
 \begin{align*}
			u_t(x,y)  \geq \frac{\widetilde C^2}{4^{2+d/2} \Gamma(d/2+1) (4\pi t)^{d/2}} & \exp\left(-\frac{dt}{2}\right) \exp\left(-\left(V^{*}(\rho_y)+\frac{\mu_0}{4\rho_y^2}\right)t\right) \\ & \times \exp\left(- \frac{9}{2} \rho_x \sqrt{V^{*}(\rho_x)+\frac{\mu_0}{4\rho_x^2}}\right).
		\end{align*}
In the last line we used the equality
\begin{equation*}
	\exp\left(-\frac{\rho_x V^*(\rho_x)}{2\sqrt{V^{*}(\rho_x)+\frac{\mu_0}{4\rho_x^2}}}\right)  \exp\left(- \mu_0\frac{t_{\rho_x}}{4\rho_x^2}\right) =  \exp\left(- \frac{1}{2} \rho_x \sqrt{V^{*}(\rho_x)+\frac{\mu_0}{4\rho_x^2}}\right).
\end{equation*}
	\end{proof}	
	
We are now ready to give the proof of the main result of this section.
	
\begin{proof}[Proof of Theorem \ref{th:lower_main}]
Let $x ,y \in \R^d$, $t>0$.  Part (2) follows directly from Lemma \ref{lem:aux-1} with $\rho=\rho_x \vee \rho_y$ and the inequality $\widetilde C /4 \geq c_2$.  We only need to establish Part (1). 

Assume that $4t_{\rho_x \vee \rho_y} \leq t$ and suppose that $\rho_x \geq \rho_y$. In particular, $t_{\rho_x \vee \rho_y} = t_{\rho_x}$.  Observe that
\[
4t_{\rho_z} \leq t \quad \Longleftrightarrow \quad  2 \rho_z \sqrt{V^{*}(\rho_z)+\frac{\mu_0}{4\rho_z^2}} \leq \left(V^{*}(\rho_z)+\frac{\mu_0}{4\rho_z^2}\right) t.
\]
In particular,
		\begin{equation*}
		4t_{\rho_z} \leq t \implies K(t,\rho_z) = \exp\left(-\frac{9}{2} \rho_z \sqrt{V^{*}(\rho_z)+\frac{\mu_0}{4\rho_z^2}}  \right)
		\end{equation*}
		and
		\begin{equation*}
		4t_{\rho_z} \geq t \implies K(t,\rho_z) = \exp\left(-\frac{9}{4} \left(V^{*}(\rho_z)+\frac{\mu_0}{4\rho_z^2}\right) t \right).
		\end{equation*}
In order to complete the proof, we need to consider two cases and just use the above observations.

Let first  $4t_{\rho_y} \leq t$.  Since $2t_{\rho_x} \leq t/2$ and $2t_{\rho_y} \leq t/2$, we can use Lemma \ref{lem:aux-2} and symmetry to get
\[
u_{t/2}(x,z) \geq \frac{(\widetilde C/4)^2}{4^{d/2}\Gamma(d/2+1)} \frac{1}{(2\pi t)^{d/2}} \exp\left(-\left(\frac{d}{2}+V^{*}(1)+\frac{\mu_0}{4}\right)\frac{t}{2}\right)  \exp\left(-\frac{9}{2} \rho_x \sqrt{V^{*}(\rho_x)+\frac{\mu_0}{4\rho_x^2}}  \right)
\]
and
\[
u_{t/2}(z,y) \geq \frac{(\widetilde C/4)^2}{4^{d/2}\Gamma(d/2+1)} \frac{1}{(2\pi t)^{d/2}} \exp\left(-\left( \frac{d}{2}+V^{*}(1)+\frac{\mu_0}{4}\right)\frac{t}{2}\right)  \exp\left(-\frac{9}{2} \rho_y \sqrt{V^{*}(\rho_y)+\frac{\mu_0}{4\rho_y^2}}  \right),
\]
whenever $|z| \leq 1$. By the Chapman--Kolmogorov property and the inequality $t^{-d/2} \geq e^{-(d/2)t}$, we then obtain the estimate
\begin{align*}
u_t(x,y) & \geq  \int_{B_1}u_{t/2}(x,z)u_{t/2}(z,y)dz \\ &  \geq \left(\frac{(\widetilde C/4)^2}{4^{d/2}\Gamma(d/2+1)}\right)^2 \frac{|B_1|}{\pi^{d/2}(4\pi t)^{d/2}} e^{-\lambda_2  t} K(t,\rho_x) K(t,\rho_y)   \\
                 &  = c_2 e^{-\lambda_2  t} g_t(0,0) K(t,\rho_x) K(t,\rho_y). 
\end{align*}
This is exactly what we wanted to prove. 

 Suppose now that $4 t_{\rho_y} \geq t$. Recall that we have $2t_{\rho_x} \leq t$ and $\rho_x \geq \rho_y$. Then, by Lemma	\ref{lem:aux-2}, we get 
 \begin{align*}
			u_t(x,y)  & \geq \frac{\sqrt{c_2}}{(4\pi t)^{d/2}}  \exp\left(-\frac{dt}{2}\right) \exp\left(-\frac{9}{4} \left(V^{*}(\rho_y)+\frac{\mu_0}{4\rho_y^2}\right)t\right) \exp\left(-\frac{9}{2}  \rho_x \sqrt{V^{*}(\rho_x)+\frac{\mu_0}{4\rho_x^2}}\right) \\
			          &  \geq c_2 \exp\left(-\lambda_2 t\right) g_t(0,0) K(t,\rho_y)  K(t,\rho_x).
		\end{align*}
 This completes the proof.  
\end{proof}

\section{Applications, previous results, discussion and examples} \label{sec:applications}

\subsection{Qualitatively sharp two-sided bounds for confining potentials} \label{subsec:conf} Recall that the Schr\"odinger operator with confining potential has compact resolvent and purely discrete spectrum, $0<\lambda_0 = \inf \sigma(H)$ is a simple eigenvalue, and the corresponding eigenfunction $\varphi_0 \in L^2(\R^d,dx)$ is continuous, bounded and strictly positive on $\R^d$. We refer to $\lambda_0$ and $\varphi_0$ as to the \emph{ground state} eigenvalue and eigenfunction. 

The framework of confining potentials provides important examples of intrinsically ultracontractive Schr\"odinger semigroups, see Davies and Simon \cite{Davies-Simon}, Davies \cite{Davies-book} and Ba\~nuelos \cite{Banuelos}. One of equivalent definitions says that the semigroup $\big\{e^{-tH}: t \geq 0\big\}$ is \emph{intrinsically ultracontractive} (IUC in short), if for every $t>0$ there exists $c = c(t)$ such that for every $x, y \in \R^d$ we have $u_t(x,y) \leq c \, \varphi_0(x) \, \varphi_0(y)$. This implies the following two-sided estimates:\ for every $t_0 >0$ there exists $\widetilde c=\widetilde c(t_0) \geq 1$ such that
\[
\frac{1}{\widetilde c} \, e^{-\lambda_0 t} \, \varphi_0(x) \, \varphi_0(y) \leq u_t(x,y) \leq \widetilde c \, e^{-\lambda_0 t} \, \varphi_0(x) \, \varphi_0(y), \qquad x, y \in \R^d, \ \ t\geq t_0.
\]
For many examples of potentials it is known that $\varphi_0$ is comparable to $\exp\big(-c |x| \sqrt{V(x)}\big)$, typically with different constant $c$ from above and from below, see Simon \cite{Simon-1} and Carmona \cite{Carmona}, and references in these papers. Clearly, this leads to qualitatively sharp two-sided estimates for large times. On the other hand, the estimates for non-IUC semigroups were an open problem for a long time. 
For example, if
\begin{align} \label{eg:conf_ex}
V(x) = |x|^{\alpha}, \qquad \alpha>0,
\end{align}
then IUC holds if and only if $\alpha >2$. 
More recently, this example has been studied by Sikora \cite{Sikora} who proved the upper bounds for the full range of $\alpha>0$ and $t>0$, and the two-sided diagonal-on estimate. Short time upper estimates of a similar type for the potential \eqref{eg:conf_ex} and more general profiles were also obtained by Metafune and Spina \cite{Metafune-Spina}, and Spina \cite{Spina}. Metafune, Pallara and Rhandi \cite{Metafune-Pallara-Rhandi} analyzed the constant in the intrinsic ultracontractivity estimate for small $t$. We also refer the reader to the paper by Ouhabaz and Rhandi \cite{Ouhabaz-Rhandi} for upper estimates for uniformly elliptic operators. 

Our Theorems \ref{th:main-1} and \ref{th:lower_main} immediately give the following, global in time and space, two-sided estimates for general nonnegative locally bounded confining potentials, in both IUC and non-IUC settings. Recall that the functions $H$ and $K$, and the constants $\gamma_1, \gamma_2, c_1, c_2$ come directly from the statements of these theorems.

\begin{corollary} \label{cor:dec} 
For every confining potential $0 \leq V \in L^{\infty}_{\loc}(\R^d)$, $x,y \in \R^d$ and $t >0$ we have the following estimates.
\begin{itemize}
\item[(1)] If $4t_{\rho_x \vee \rho_y} \leq t$, then
\begin{align*}
c_2  e^{-\gamma_2 t}\, K(t,\rho_x) \, K(t,\rho_y) \, g_t(0,0) \leq u_t(x,y)
	    \leq c_1\, e^{-\gamma_1 t}  \, H(t,x) \,  H(t,y)\, g_{t}(0,0).
\end{align*}
\item[(2)] If $4t_{\rho_x \vee \rho_y} \geq t$, then
\begin{align*}
c_2 \, K(t,\rho_x) \, K(t,\rho_y) \, g_t(x,y) \leq u_t(x,y)  
	     \leq c_1 \, H(t,x) \,  H(t,y)\, g_{2t}(x,y).
\end{align*}
\end{itemize}
\end{corollary}

These estimates take the sharpest form for potentials comparable to radial monotone functions that grow at infinity sufficiently regularly.
They are fully uniform if the growth of the potential is not too fast. We give an illustration with the following two examples. Example \ref{ex:conf_gen} is general, and Example \ref{ex:conf} shows the applications to some specific classes of potentials.

\begin{example} [\textbf{Potentials with radial monotone profiles}] \label{ex:conf_gen}
Let $0 \leq V \in L^{\infty}_{\loc}(\R^d)$ be such that there exists a constant $m \geq 1$ satisfying
\begin{align} \label{eq:doubling}
V^*(|x|) \leq m V_*(x) , \qquad |x| \geq 1.
\end{align}
Define $W(r):=V^*(r)$. Then $W$ is referred to as the radial monotone profile of the potential $V$. 

It is convenient to introduce the following rate functions:
\[
\widetilde K(t,x) := \exp\left(- \frac{9}{4}\left( \left( W(\rho_x) +\frac{\mu_0}{4\rho_x^2}\right)t \, \wedge \,  2\rho_x \sqrt{W(\rho_x) + \frac{\mu_0}{4\rho_x^2}}\right)\right) \quad \text{for} \quad \rho_x = |x| \vee 1,
\]
\smallskip
\[
\widetilde H(t,x) := \exp\left(- \frac{\sqrt{2}}{32 m }\left( \left(W(|x|) +\frac{\mu_0}{4|x|^2}\right)t \, \wedge \,  2|x| \sqrt{W(|x|)+\frac{\mu_0}{4|x|^2}}\right)\right) \quad \text{for} \quad |x| \geq 1,
\]
$\widetilde H(t,x) = 1$ for $|x| < 1$. Moreover, let
 \[
\gamma_1= \frac{\lambda_0}{2},  \qquad \qquad \gamma_2=d+W(1) +\frac{\mu_0}{4}.
\] 
Observe that $\widetilde K(t,x)$ and $\widetilde H(t,x)$ take exactly the same form 
\[
\exp\left(- c \left( \left(W(|x|) +\frac{\mu_0}{4|x|^2}\right)t \, \wedge \,  2|x| \sqrt{W(|x|)+\frac{\mu_0}{4|x|^2}}\right)\right) \quad \text{for} \quad |x| \geq 1, \ t>0,
\]
and they differ only by the value of the constant $c$ in the exponent. Moreover, we always have
\[
0<\exp\left(- \frac{9}{2}\sqrt{W(1)+\frac{\mu_0}{4}}\right) \leq \widetilde K(t,x) \leq \widetilde H(t,x) = 1, \qquad |x| < 1, \ t>0.
\]
We obtain from Corollary \ref{cor:dec} the following qualitatively sharp uniform two-sided estimates:
\begin{itemize}
\item[(1)] If \[\frac{2(\rho_x \vee \rho_y)}{\sqrt{W(\rho_x \vee \rho_y) +\frac{\mu_0}{4(\rho_x \vee \rho_y)^2}}} \leq t,\] then
\begin{align*}
c_2  e^{-\gamma_2 t}\, \widetilde K(t,x) \, \widetilde K(t,y) \, g_t(0,0) &\leq u_t(x,y)
	   \leq c_1\, e^{-\gamma_1 t}  \, \widetilde H(t,x) \, \widetilde H(t,y)\, g_{t}(0,0).
\end{align*}
\item[(2)] If \[\frac{2(\rho_x \vee \rho_y)}{\sqrt{W(\rho_x \vee \rho_y) +\frac{\mu_0}{4(\rho_x \vee \rho_y)^2}}} \geq t,\] then
\begin{align*}
c_2 \, \widetilde K(t,x) \, \widetilde K(t,y) \, g_t(x,y) \leq u_t(x,y)  
	     \leq c_1 \, \widetilde H(t,x) \, \widetilde H(t,y)\, g_{2t}(x,y).
\end{align*}
\end{itemize}
\end{example}

We now apply Example \ref{ex:conf_gen} to specific potentials. 

\begin{example} [\textbf{Polynomial and logarithmic potentials}] \label{ex:conf} Consider the following classes of potentials.

\begin{enumerate} 
\item[(1)] \emph{Polynomial potentials}:\ Let
\begin{align*} \label{eg:conf_ex_k}
V(x) = k |x|^{\alpha}, \qquad \alpha, k >0.
\end{align*}
Clearly, we have 
\[
V_*(x) = k (|x|/2)^{\alpha} \qquad \text{and} \qquad V^*(r) = k(2r)^{\alpha}.
\] 
Moreover, $\gamma_2=d+ k \cdot 2^{\alpha} +\mu_0/4$. Observe that \eqref{eq:doubling} is true for $m = 4^{\alpha}$, uniformly in $k>0$. Consequently, the estimates from Example \ref{ex:conf_gen} hold with 
\[
W(r)=k(2r)^{\alpha} \qquad \text{and} \qquad m = 4^{\alpha}.
\]
Moreover, the rate $9/4$ in the function $\widetilde K(t,x)$ (the lower bound) is uniform in $\alpha >0$ and $k>0$, and the rate 
$\sqrt{2}/(32m)$ in $\widetilde H(t,x)$ (the upper bound) is uniform in $k>0$ and $\alpha \in (0,\alpha_0]$, for every fixed $\alpha_0>0$ -- it can be chosen as $\sqrt{2}/(32 \cdot 4^{\alpha_0})$. 

Our result applies to both IUC ($\alpha >2$) and non-IUC ($\alpha \in (0,2]$) cases. In the non-IUC setting, the estimates are fully uniform in $k>0$ and $\alpha \in (0,2]$. 
 
 \smallskip

\item[(2)] \emph{Logarithmic potentials}:\ Let
\begin{align*} \label{eg:conf_ex_k}
V(x) = \log^{\alpha}(2+ k|x|), \qquad \alpha, k >0.
\end{align*}
One has 
\[
V_*(x) = \log^{\alpha}(2+ k(|x|/2)), \qquad  V^*(r) = \log^{\alpha}(2+ k(2r)),
\] 
and $\gamma_2=d+ \log^{\alpha}(2+2k) + \mu_0/4$. Moreover, it is direct to check that the condition \eqref{eq:doubling} holds with $m = 3^{\alpha}$, uniformly in $k>0$. Consequently, we obtain the estimates as in Example \ref{ex:conf_gen}, with 
\[
W(r)=\log^{\alpha}(2+ k(2r)) \qquad \text{and} \qquad m = 3^{\alpha},
\]
In particular, the rate $9/4$ in the function $\widetilde K(t,x)$ (the lower bound) is uniform in $\alpha >0$ and $k>0$, and the rate 
$\sqrt{2}/(32m)$ in $\widetilde H(t,x)$ (the upper bound) is uniform in $k>0$ and $\alpha \in (0,\alpha_0]$, for every fixed $\alpha_0>0$ -- we can just take $\sqrt{2}/(32 \cdot 3^{\alpha_0})$. 

Note that such potentials lead to non-IUC semigroups for every $\alpha>0$. Our estimates are fully uniform in $k>0$ and $\alpha \in (0,\alpha_0]$, for every fixed $\alpha_0>0$.  
\end{enumerate}

\end{example}

We remark that one of our primary motivations to perform the present project was to understand the large time properties of the Schr\"odinger semigroups with confining potentials which are not IUC. This is related to the recent progress in the field of nonlocal Schr\"odinger operators. First note that in order to describe the large time regularity of the semigroup, it is enough to study the \emph{asymptotic} version of IUC (aIUC in short) which is more general, see \cite{Kaleta-Lorinczi} for more details. For sharp necessary and sufficient conditions for (a)IUC in the nonlocal case we refer the reader to Kulczycki and Siudeja \cite{Kulczycki-Siudeja}, Kaleta and Kulczycki \cite{Kaleta-Kulczycki}, Kaleta and L\H orinczi \cite{Kaleta-Lorinczi}, and Chen and Wang \cite{Chen-Wang-1,Chen-Wang-2} (see also the related important paper by Kwaśnicki \cite{Kwasnicki} for stable semigroups on unbounded sets). 

In \cite{Kaleta-Schilling}, Kaleta and Schilling observed in the nonlocal setting that the Schr\"odinger semigroup with cofining potential which is not (a)IUC (no matter how slow is the growth of $V$ at infinity!) still manifests a weaker version of the regularity, which can be described as follows:\ there exists an increasing function $\rho$ (determined by the kinetic term and the potential) such that $\rho(t) \uparrow \infty$ as $t \uparrow \infty$ and a constant $c>0$ such that
\[
u_t(x,y) \leq c e^{-\lambda_0 t} \varphi_0(x) \varphi_0(y), \qquad |x| \wedge |y| \leq \rho(t), \  t \geq t_0 
\] 
($c$ is uniform in $t$ and $x,y$). This was called \emph{progressive} IUC (pIUC in short). 

Our estimates proven in this paper show that the same is true for classical Schr\"odinger semigroups, at least in a qualitatively sharp fashion. A similar qualitative property can also be derived from the estimates obtained by Chen and Wang in the recent interesting preprint \cite{Chen-Wang}, which we discuss in Remark \ref{rem:CW} below.

\subsection{Upper estimate for decaying potentials} \label{subsec:dec} Our result immediately gives the upper estimate for \emph{decaying} potentials, i.e.\ when $V(x) \to 0$ as $|x| \to \infty$. For clarity, we illustrate this with the potential
\begin{align} \label{eq:dec_ex}
V(x) = k (1 \vee |x|)^{-\alpha}, \qquad \alpha, k > 0
\end{align} 
(we remark that for the upper bound it is enough to assume that $V$ is bounded from below by the expression on the right hand side of \eqref{eq:dec_ex}). The estimates for potentials of this type were obtained in a more general settings of manifolds by Zhang \cite{Zhang-1}. We compare these results with our upper estimate. 

We have
\begin{equation*}
	V_*(x) = k \cdot \begin{cases}
		1 &  \text{for} \ |x| \leq \frac{2}{3} \\
		\big( \frac{3}{2} |x|\big)^{-\alpha}&\text{for} \ |x| \geq \frac{2}{3}.
	\end{cases}
\end{equation*}
The following estimate can be easily derived from Theorem \ref{th:main-1}. 

\begin{example} \label{ex:dec} Let $V$ be as in \eqref{eq:dec_ex}. For every $x, y \in \R^d$, $t>0$ have the following upper bound: 
\[
u_t(x,y) \leq c_1 \widetilde H(t,x) \widetilde H(t,y) g_{2t}(x,y),
\]
where $\widetilde H(t,x) = 1$ for $|x| < 1$ and
\[
\widetilde H(t,x) := \begin{cases}
 		\exp\left(- \frac{\sqrt{2}}{32} \cdot \left(\frac{2}{3}\right)^{\alpha} \left( \frac{k t}{|x|^{\alpha}} \wedge \,  2 \sqrt{k} |x|^{1-\alpha/2}\right)\right)	& \text{if} \ \alpha \in (0,2), \\
 		1	& \text{if} \ \alpha \geq 2, 
 	\end{cases}
		\qquad \text{for} \ |x| \geq 1.
\]
\end{example}
The rate in the exponent is an explicit constant which is uniform in the coupling parameter $k>0$ and $\alpha \in (0,\alpha_0]$, for every fixed $\alpha_0>0$.

As explained in Section \ref{sec:results}, our Theorem \ref{th:lower_main} is not sharp enough to give a similar lower bound for decaying potentials. This is a much more difficult problem that requires much subtle argument. On the other hand, this example shows that at least  for the potentials as in \eqref{eq:dec_ex} with $\alpha \in (0,2)$, the function $\widetilde H(t,x)$ resulting from Theorem \ref{th:main-1} is sharp in one of the regimes. For small $t$'s the kernel $u_t(x,y)$ is just comparable to $g_t(x,y)$, so we only look at large times. Indeed, if $t \leq (2/\sqrt{k}) |x|^{1+\alpha/2}$, then $\widetilde H(t,x) = \exp\big(- c k t/|x|^{\alpha}\big)$,
which is qualitatively the same as the function $w_2(x,t)$ in the lower estimate of \cite[Theorem 1.2]{Zhang-1}; this means that we improve the power in the exponent of the function $w_1(x,t)$ in \cite[Theorem 1.1]{Zhang-1} for this time-space region (see the comments in Remark 1.2 of the quoted paper). In fact, the exponent in our function $\widetilde H(t,x)$ is qualitatively better in the larger time-space region which is roughly described by $\sqrt{t} \leq |x|^{1+\alpha/2}$. We also remark that due to Theorem \ref{th:upper_final} the Gaussian term in the estimate above can be made  nearly optimal  at the cost of the rate in $\widetilde H$, and we have an uniform control with respect to parameter $k>0$ and $\alpha \in (0,\alpha_0]$. On the other hand, if $\sqrt{t} \geq |x|^{1+\alpha/2}$, then the estimate in \cite[Theorem 1.1]{Zhang-1} is sharper than ours. For $\alpha>2$ the estimate above is trivial, but is also sharp. 

\subsection{Upper estimate for potentials bounded away from zero.} \label{subsec:bdd_away} We can also give a non-trivial upper estimate for potentials that are \emph{bounded away from zero outside a bounded set}, i.e.\ the functions $V$ for which
\begin{align} \label{eq:bdd_away_ex}
\text{there exist} \  \kappa >0 \ \text{and} \ r_0 \geq 0 \ \text{such that} \  V(x) \geq \kappa, \ |x| \geq r_0.
\end{align} 
Theorem \ref{th:main-1} immediately gives the following estimate. 

\begin{example} \label{ex:bdd} Let $V$ be as in \eqref{eq:bdd_away_ex}. For all $x, y \in \R^d$, $t>0$ we have the following upper bound: 
\[
u_t(x,y) \leq c_1 \widetilde H(t,x) \widetilde H(t,y) g_{2t}(x,y),
\]
where $\widetilde H(t,x) = 1$ for $|x| < 2r_0$ and
\[
\widetilde H(t,x) = \exp\left(- (\sqrt{2}/32) \left( \kappa t \wedge \,  2\sqrt{\kappa} |x|\right)\right),
		\qquad \text{for} \ |x| \geq 2r_0.
\]
Moreover, if $r_0=0$ or we just know that $\lambda_0>0$, then $g_{2t}(x,y)$ can be replaced with $e^{-(\lambda_0/2) t} g_{t}(0,0)$.
\end{example}
\subsection{Upper estimate for more general, non-radial potentials.} \label{subsec:general}

Observe that by using Theorem \ref{th:main-1} we can also get the upper estimate of the same type for highly non-radial potentials, including confining, decaying, bounded away from zero ones, and mixtures. This follows from the fact that $H(t,x)$ depends only on the values of the potential $V$ in the ball $\overline{B}_{|x|/2}(x)$. It seems to be a novelty even if $d=1$. Let $\alpha_1, \alpha_2, c>0$. For a quick overview, we just list some examples of potentials on $\R$: 
\begin{enumerate}
\item[(1)] (nonsymmetric confining potential):
\begin{equation*}
	V(x) = \begin{cases} 
	 x^{\alpha_1}, & \ \ x \geq 0, \\
	 (-x)^{\alpha_2}, & \ \ x \leq 0,
		\end{cases}
		\qquad \text{with} \qquad \alpha_1 \neq \alpha_2.
\end{equation*}

\item[(2)] (nonsymmetric decaying potential):
\begin{equation*}
	V(x) = \begin{cases} 
	 \ (1 \vee x)^{-\alpha_1}, & \ \ x \geq 0, \\
	 (1 \vee -x)^{-\alpha_2}, & \ \ x \leq 0,
		\end{cases} 
		\qquad \text{with} \qquad\alpha_1 \neq \alpha_2.
\end{equation*}

\item[(3)] (mixture of confining and decaying potentials):
\begin{equation*}
	V(x) = \begin{cases} 
	  (1 \vee x)^{-\alpha_1}, & \ \ x \geq 0, \\
	 (-x)^{\alpha_2}, & \ \ x \leq 0,
		\end{cases}
		\qquad \text{or} \qquad 
			V(x) = \begin{cases} 
	 x^{\alpha_1}, & \ \ x \geq 0, \\
	 (1 \vee -x)^{-\alpha_2}, & \ \ x < 0.
		\end{cases}
\end{equation*}

\item[(4)] (mixture of confining/decaying and constant potentials):
\begin{equation*}
	V(x) = \begin{cases} 
	 c, & \ \ x \geq 0, \\
	 (-x)^{\alpha_1}, & \ \ x <  0,
		\end{cases}
		\qquad \text{or} \qquad 
			V(x) = \begin{cases} 
	 	(1 \vee x)^{-\alpha_1},& \ \ x > 0, \\
     c & \ \ x  \leq  0,
		\end{cases}
\end{equation*}
or the other rearrangements.
\end{enumerate}

For all of the examples of this type we obtain the upper bound similar to those in Examples \eqref{ex:conf}, \eqref{ex:dec} and \eqref{ex:bdd} above, but with $\widetilde H$ which takes a different form for $x>0$ and $x<0$.

\begin{remark} \label{rem:CW} 
While finalizing the present manuscript, J. Wang informed us 
about his recent preprint \cite{Chen-Wang} with X. Chen, where the authors provide 
the two-sided qualitatively sharp
estimates for the heat kernels of Schrödinger operators with 
nonnegative, locally bounded,
confining potentials V, which are comparable to radial and monotone profiles $g$. This 
result is related to our Examples \ref{ex:conf} and \ref{ex:conf_gen}
which apply to the same class of potentials.

We want to point out that our work was performed simultaneously with, 
and independently of, \cite{Chen-Wang}. The estimates in \cite{Chen-Wang} have a completely 
different structure from ours, which seems to be related to the fact 
that the argument in \cite{Chen-Wang} requires the assumption that $s \mapsto (1+s)/\sqrt{g(s)}$ 
is an almost monotone function. Note that our Theorem \ref{th:main-1}
covers a larger class of potentials, as it applies to general
nonnegative locally bounded potentials.
The arguments in \cite{Chen-Wang} and in our paper are based on completely different 
ideas. The key step in our proof of the upper estimate is based on the 
fact that the Laplace transform
of the exit time from a ball with radius proportional to the norm of 
$x$ and evaluated at $\lambda = V(x)$ takes the shape of the ground state $\varphi_0(x)$ of 
the Schr\"odinger operator $H$. Here we
apply the classical result of Wendel \cite{Wendel}. In the proof of the lower 
estimate, we use directly
the estimate (with optimal Gaussian term) for the semigroup of a 
Brownian motion in a ball, which was recently discovered by Małecki and 
Serafin \cite{malecki-serafin}, and combine this bound with the direct
estimates in Lemma \ref{lem:aux-2}. This is described in more detail at the end of Section \ref{sec:results}.
Our proof does not use any information on the joint distribution of the 
exit position and the exit time of
Brownian motion from a ball, which is the main tool in \cite{Chen-Wang}. Our approach 
leads us to
qualitatively sharp estimates with explicit numerical rates, uniform in 
$V$, with clear dependence
on the dimension $d$, and optimal or  nearly optimal  Gaussian terms, which 
are obtained by
rather short and direct proofs. The constants in the estimates of \cite{Chen-Wang} 
are not explicit; they seem
to depend of $V$ and $d$ in an implicit fashion. 
\end{remark}

\subsection*{Data availability} Data sharing not applicable to this article as no datasets were generated or analysed during the current study.

\subsection*{Conflict of interests statement} The authors have no competing interests to declare that are relevant to the content of this article.

\bibliographystyle{plain}

\end{document}